\documentclass{article}

\usepackage{hyperref}
\usepackage[bibtex-style]{amsrefs}
\usepackage{amsmath, amsthm, amssymb, amsfonts}
\usepackage{bm}
\usepackage{tikz}
\usetikzlibrary{calc}
\usetikzlibrary{decorations}
\usetikzlibrary{decorations.pathreplacing}
\usepackage{xspace}
\newcommand{\Av}{\mathop{\mathrm{Av}}}
\newcommand{\we}{\ensuremath{\simeq}\xspace}

\newcommand{\A}{\ensuremath{{\mathcal A}}\xspace}
\newcommand{\B}{\ensuremath{{\mathcal B}}\xspace}
\newcommand{\C}{\ensuremath{{\mathcal C}}\xspace}
\newcommand{\D}{\ensuremath{{\mathcal D}}\xspace}

\newcommand{\bigclass}{\ensuremath{{\mathcal M}}\xspace}

\newcommand{\archOver}[1]{%
\,
\begin{tikzpicture}[anchor=base,baseline,scale=0.3]
\draw (0,0) node[inner sep=0pt] (a) {$#1$};
\coordinate (A) at ($ (a.south west) + (-0.2,0) $);
\coordinate (B) at ($ (a.north west) + (-0.2,0.2) $);
\coordinate (C) at ($ (a.north east) + (0.2,0.2) $);
\coordinate (D) at ($ (a.south east) + (0.2,0) $);
\draw[rounded corners=4pt] (A) -- (B) -- (C) -- (D);
\end{tikzpicture}
\,
}
\newcommand{\archOverSubscript}[1]{%
\,
\begin{tikzpicture}[anchor=base,baseline,scale=0.2,font=\scriptsize]
\draw (0,0) node[inner sep=0pt] (a) {$#1$};
\coordinate (A) at ($ (a.south west) + (-0.2,0) $);
\coordinate (B) at ($ (a.north west) + (-0.2,0.2) $);
\coordinate (C) at ($ (a.north east) + (0.2,0.2) $);
\coordinate (D) at ($ (a.south east) + (0.2,0) $);
\draw[rounded corners=4pt] (A) -- (B) -- (C) -- (D);
\end{tikzpicture}
\,
}
\newcommand{\emptyArch}{\archOver{\rule{0pt}{1ex}\hspace{1ex}}}

\DeclareMathOperator{\mset}{\textsc{MSet}}

\newtheorem{theorem}{Theorem}

\newtheorem{conjecture}[theorem]{Conjecture}
\newtheorem{proposition}[theorem]{Proposition}
\newtheorem{definition}[theorem]{Definition}
\newtheorem{observation}[theorem]{Observation}
\newtheorem{remark}[theorem]{Remark}

\title{A general theory of Wilf-equivalence for Catalan structures}
\author{Michael Albert \and Mathilde Bouvel}
\date{}

\begin{document}
\maketitle

\begin{abstract}
The existence of apparently coincidental equalities (also called Wilf-equivalences) between the enumeration sequences, or generating functions, of various hereditary classes of combinatorial structures has attracted significant interest. We investigate such coincidences among non-crossing matchings and a variety of other Catalan structures including Dyck paths, 231-avoiding permutations and plane forests. In particular we consider principal classes defined by not containing an occurrence of a single given structure. An easily computed equivalence relation among structures is described such that if two structures are equivalent then the associated principal classes have the same enumeration sequence. We give an asymptotic estimate of the number of equivalence classes of this relation among structures of size $n$ and show that it is exponentially smaller than the $n^{\mbox{\scriptsize th}}$ Catalan number. In other words these ``coincidental'' equalities are in fact very common among principal classes. 
Our results also allow us to prove, in a unified and bijective manner, several known Wilf-equivalences from the literature. 

\end{abstract}

\section{Introduction}

The Catalan numbers are renowned for their ubiquity in problems of combinatorial enumeration. 
A few of the many contexts in which they arise are: plane forests (counted by number of nodes), 
non-crossing matchings or arch systems (counted by number of matched pairs or arches), Dyck paths, and 231-avoiding permutations. 
These contexts share the additional property -- to be detailed in Section~\ref{sec:CatalanStructures} -- 
that each admits a natural substructure relation, and that there are bijections between them which preserve that relationship. 
So, one can further consider those structures of each type which do not contain some designated substructure(s). 
As part of a previous work (see an extended abstract \cite{FPSAC2013},  or \cite{ABSorting}) 
the present authors considered certain coincidences of enumeration (often called \emph{Wilf-equivalences}) 
between such classes of Catalan structures avoiding a given substructure (in our case, permutations avoiding 231 and $\pi$). 
Using a non-standard bijection we were able to explain some of those coincidences. However, when we turned to the more general question:
\begin{quote}
\emph{How many distinct enumeration sequences are there for classes of 231-avoiding permutations defined by a single additional restriction?}
\end{quote}
we were struck by the difference between the computed numbers, and any known general equivalences. 
Specifically it seemed that there were many more such coincidences (and so fewer enumeration sequences) than one might have expected. 
This phenomenon will be explained in the current paper. 
We will show in Section~\ref{sec:cohorts} that although there are $\textrm{Cat}_n = {2n \choose n}/(n+1) \sim (1/\sqrt{\pi} ) n^{-3/2} 4^n$ distinct classes 
of permutations avoiding 231 and an additional permutation of size $n$, 
these classes have asymptotically at most $c n^{-3/2} \gamma^n$ distinct enumeration sequences where $c \approx 1.13$ and $\gamma \approx 2.4975$ 
(these are approximate values only). 

A particularly wide collection of such classes share generating functions derived from the continued fraction representation of $C(t) = \sum \textrm{Cat}_n t^n$, 
the generating function of the Catalan numbers. Since $C = 1/(1 - tC)$ it follows that:
\[
C = \cfrac{1}{1 - \cfrac{t}{1 - \cfrac{t}{1 - \cfrac{t}{1 - \cdots}}}}
\]
This fraction can be truncated after $n$ levels, producing a sequence of generating functions:
\begin{align*}
C_0 &= 1 \\
C_n &= \frac{1}{1-t \, C_{n-1}} \quad \mbox{for $n \geq 1$}
\end{align*}
The functions $C_n$ enumerate many specific subclasses of the Catalan classes above -- 
for instance the 231-avoiding permutations that also avoid a descending permutation of size $n$, 
or the Dyck paths of height at most $n$. 
Other examples can be found in \cites{Mansour:Chebyshev, FPSAC2013}. 
Previously these enumeration coincidences were understood on an analytic (or perhaps more properly arithmetic) level only. 
We can explain them, and many others, bijectively -- among other things we can show, 
combining Propositions~\ref{pr:MotzkinManyInMainCohort}, \ref{pr:GFofMainCohort} and~\ref{pr:easiestToAvoid}:
\begin{quote}
\emph{The number of 231-avoiding permutations, $\pi$, of size $n$ 
for which the generating function of the class of permutations avoiding both 231 and $\pi$ 
is $C_n(t)$ is the $n^{\mbox{\scriptsize{th}}}$ Motzkin number.}
\end{quote}
The proof of this fact also describes (at least in principle) bijections between any two such classes. 
Furthermore, we show that for any other 231-avoiding permutation $\theta$ of size $n$, 
the generating function for 231 and $\theta$-avoiding permutations is dominated (term by term and eventually strictly) by $C_n(t)$.

The main tool in producing these results is a binary relation on Catalan structures defined purely intrinsically by four very simple rules in Section~\ref{sec:equivalenceRelation}. 
This relation induces an equivalence relation $\sim$ on these Catalan structures whose equivalence classes are the connected components of the binary relation. 
Remarkably,  if $A \sim B$ then the collection of structures not containing $A$ has the same generating function as the collection of structures not containing $B$, 
so that one generating function may be associated with each equivalence class of $\sim$. 
For convenience in the description and proofs we will work mostly in the domain of arch systems, 
but of course all the results translate to the other domains directly using the natural bijections of Section~\ref{sec:CatalanStructures}. 
We have been able to verify that through size 15 (where $\sim$ has only 16,709 equivalence classes on the 9,694,845 Catalan structures) 
that $A$ and $B$ are $\sim$-equivalent if and only if the corresponding generating functions are the same. So we have:

\begin{conjecture}
\label{co:cohortDeterminesGF}
The equivalence relation $\sim$ coincides with Wilf-equivalence.
\end{conjecture}
In the final section we discuss this conjecture, and further open problems.

In the next section we consider the quartet of Catalan structures, namely arch systems, Dyck paths, plane forests, and 231-avoiding permutations in more detail 
and introduce our basic terminology and notation. 
This is followed by some preparatory results before we introduce the relation $\sim$ and prove its main property, 
namely that it refines Wilf-equivalence in Theorem \ref{th:mainTheorem}. 
We can represent the collection of all $\sim$-equivalence classes, which we call \emph{cohorts}, as a slight modification of the family of non-plane forests 
and this also permits us to determine the number of cohorts in structures of size $n$, both through a functional equation or recurrence and asymptotically. 
We then consider further relationships between the cohorts, 
and the properties of the special main cohort mentioned above -- which is maximal in terms of the associated generating functions 
and also conjecturally in terms of the cardinality of the cohort. 
Finally we consider some open problems that arise from this work.

\section{Arch systems, Dyck paths, plane forests, and 231-avoiding permutations}
\label{sec:CatalanStructures}

Among the most well-known Catalan structures are certainly the Dyck paths. 
A \emph{Dyck path} of semi-length $n$ is a path in the positive quarter-plane, taking steps $u= (1,1)$ and $d= (1,-1)$, starting at $(0,0)$ and ending at $(2n,0)$. 
Steps $u$ and $d$ of a Dyck path may be paired, by associating to each $u$ step the first $d$ step on its right at the same ordinate. 
These pairs $(u,d)$ may also be seen as pairs of opening and closing parentheses, 
and under this correspondence Dyck paths correspond to parentheses word where parentheses are properly matched. 
A subpath of a Dyck path is defined by the deletion of some pairs of steps $(u,d)$ (or equivalently of matched parentheses). 
The deletion here is intended as a \emph{contraction} of the segment of each deleted step into a point, 
so that deleting $k$ pairs of steps in a Dyck path of semi-length $n$ provides a Dyck path of semi-length $n-k$. 

Another natural way of representing proper parentheses words is as non-crossing matchings or arch systems. 
These form a second family of Catalan structures, and will be essential in the presentation of our results. 
An \emph{arch system} of size $n$ is a set of $n$ arches connecting $2n$ points arranged along a baseline, 
such that all arches are above the baseline and no pair of arches cross. 
The left end of each arch encodes an opening parenthesis and its right end the corresponding closing parenthesis. 
A subsystem of an arch system can be obtained simply by deleting some of the original system's arches. 

We can concatenate arch systems, $A$ and $B$ in the obvious way -- just draw the arch system $B$ strictly to the right of $A$ on the same baseline. 
The resulting arch system will be denoted $AB$.

\begin{definition}
An \emph{atom} is a non empty arch system that cannot be written as the concatenation of two non empty arch systems, 
i.e.~one that has a single outermost arch. Atoms will generally be denoted by lower case letters. 
The \emph{contents} of an atom $a$ are the unique arch system, $A$, such that $a$ is obtained by adding a single arch outside all of $A$, and we write $a = \archOver{A}$.
\end{definition}

Since every non empty arch system is a unique concatenation of atoms, 
we see immediately that the generating function for arch systems, $A(t)$ according to the number of arches satisfies:
\begin{align*}
A(t) &= 1 + t \, A(t) + (t \, A(t) )^2 + (t \, A(t) )^3 + \cdots \\
{} &= \frac{1}{1 - t \, A(t)}
\end{align*}
proving that -- and this should be no surprise -- that arch systems are enumerated by the Catalan numbers.

There is a bijection between arch systems with $n$ arches, and non-empty plane forests with $n$ nodes 
obtained simply by mapping each arch to a node in such a way that if one arch lies within another, then its node is a descendant of the other, 
and if it lies to the left of another, then its node does so too. 
Equivalently, describing this recursively: take an arch system $A$, write it as a concatenation of atoms $A = a_1 a_2 \cdots a_m$ 
and associate to it a forest of $m$ trees whose roots, $r_i$,  correspond to the outermost arches of the $a_i$ 
(and are arranged from left to right for $i$ from 1 through $m$) and such that the tree rooted at $r_i$ is (up to the addition of the root $r_i$) the forest of the contents of $a_i$. 
This bijection also preserves the ``substructure'' relationship provided that in the case of forests we maintain ancestry in substructures 
(e.g.~if a child, $x$, of a node, $y$, is deleted, then all the children of $x$ remaining become children of $y$, 
preserving their left to right order both among themselves and with respect to their new siblings).

Finally, we can consider 231-avoiding permutations of $\{1, 2, \dots, n\}$. 
These are those permutations $\pi$ which, when written in one line notation, contain no subsequence $bca$ with $a < b < c$. 
Here the substructure relationship (known as the \emph{pattern} relationship among permutations) 
involves deleting some symbols and then relabelling the remaining ones to form a permutation of $\{1,2,\dots,m\}$ for some $m < n$ 
while maintaining their relative order (e.g.~if we delete 2 from 31254 we obtain 2143). 
It is perhaps not immediately clear that these are also in bijection with Dyck paths, arch systems or plane forests. 
However, these permutations are precisely those that can be sorted by a single pass through a stack \cite{Knuth:Art} 
and we can form a Dyck path by adding a step $u$ whenever pushing an element on to the stack, and a step $d$ whenever popping one from the stack. 
Since the sequence of push and pop operations to sort a permutation is easily seen to be unique, and every sequence of operations sorts \emph{some} permutation 
this is clearly a bijection. Moreover, it respects the substructure relationship since, 
when deleting an element, we just delete the pair of matched steps, 
or equivalently the arch in the corresponding arch system, 
which corresponds to push and pop operations that affect that element. 
This bijection can also be realised intrinsically. 
The $n$ arches are labelled with the integers from $1$ through $n$ according to the following rules: 
if two arches are nested, then the outer arch has a greater label than the inner one, 
and if two arches are not nested the arch to the left has a lesser label than the arch to the right. 
The permutation is then read by reading the labels of the arches in order of their leftmost endpoints. 
This means that the left to right maxima of the permutation (i.e. the elements that have no greater element to their left) correspond to outermost arches, 
and within them an arch system is constructed using the same principle recursively on the following lesser elements.
An example of these correspondences is given in Figure \ref{fi:archForestPermExample}.

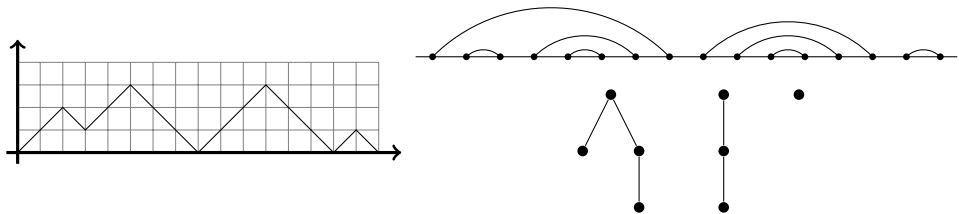
\begin{figure}

\begin{minipage}[c]{0.45\textwidth}
\centerline{
\begin{tikzpicture}[scale=0.3]
\draw[help lines] (0,0) grid (16,4);
\draw[very thick,->] (-0.5,0) -- (17,0);
\draw[very thick,->] (0,-0.5) -- (0,5);
\draw (0,0) -- (1,1) -- (2,2) -- (3,1) -- (4,2) -- (5,3) -- (6,2) -- (7,1) -- (8,0) -- (9,1) -- (10,2) -- (11,3) -- (12,2) -- (13,1) -- (14,0) -- (15,1) -- (16,0);
\end{tikzpicture}
}
\end{minipage} \begin{minipage}[c]{0.45\textwidth}
\centerline{
\begin{tikzpicture}[scale=0.45]
\draw (-0.5,0) -- (15.5,0);
\foreach \x in {0,1,...,15} 
  \fill (\x, 0) circle (0.1);
\draw (0,0)  arc[radius = 7/sqrt(2), start angle=135, end angle=45];
\draw (1,0)  arc[radius = 1/sqrt(2), start angle=135, end angle=45];
\draw (3,0)  arc[radius = 3/sqrt(2), start angle=135, end angle=45];
\draw (4,0)  arc[radius = 1/sqrt(2), start angle=135, end angle=45];
\draw (8,0)  arc[radius = 5/sqrt(2), start angle=135, end angle=45];
\draw (9,0)  arc[radius = 3/sqrt(2), start angle=135, end angle=45];
\draw (10,0)  arc[radius = 1/sqrt(2), start angle=135, end angle=45];
\draw (14,0)  arc[radius = 1/sqrt(2), start angle=135, end angle=45];
\end{tikzpicture}
}
\vspace{10pt}
\centerline{
\tikzstyle {every node} = [fill,circle,minimum size=4, inner sep = 0]
\begin{tikzpicture}[shape=circle, scale = 0.5]
\node at (0,0) {} 
  child{
    node {}
  }
  child{
     node {}
       child{
       	node {}
	}
  }
;
\node at (3,0) {}
  child{
    node {}
      child{
        node{}
      }
    };
\node at (5,0) {};
\end{tikzpicture}
}
\end{minipage} 
\caption{The Dyck path, arch system and plane forest corresponding to the 231-avoiding permutation 41327658. 
Reading $u$ steps (or arch beginnings) as push operations and $d$ steps (or arch ends) as pop operations on a stack, 
the Dyck path (or arch system) successfully sorts this permutation i.e.~the output sequence would be 12345678.}
\label{fi:archForestPermExample}
\end{figure}

\begin{remark}
Of course, there are also classical bijections between Dyck paths, plane forests or 231-avoiding permutations and plane binary trees. 
However, it is deliberate that we do not consider binary trees among the Catalan families of this work, 
since the substructure relation on Dyck paths, plane forests or 231-avoiding permutations does not translate naturally to the context of binary trees. 
This fact somehow explains why the link between 231-avoiding permutations and binary trees with respect to pattern avoidance 
is not as natural as one might hope for -- see \cite{Pudwell:Trees}*{Section 6}.
\end{remark}

In these four equivalent contexts we are interested in considering the problem:
\begin{quote}
\emph{Given a single structure $A$, what is the generating function of the collection of structures that do not have $A$ as a substructure?}
\end{quote}

Going back to some examples discussed in the introduction, 
note that Dyck paths of height at most $n$ corresponds to Dyck paths that do not have $u^nd^n$ as a subpath. 
Under the correspondences we have described, these correspond to 
arch systems that do not have $N_n = \archOver{... \archOver{\emptyArch} ... }$, the nested arch system with $n$ arches, as a subsystem, 
plane forests of depth at most $n$, 
and 231-avoiding permutations with no $n(n-1) \dots 21$ pattern. 

Structures that do not have $A$ as a substructure are said to \emph{avoid} $A$ and we will denote the set (or \emph{class}) of them by $\Av(A)$. 
If a structure does not avoid $A$ it is said to \emph{involve} or \emph{contain} $A$. 
In this paper we will only be considering the avoidance of a single structure -- 
but of course in general we could consider any collection of structures closed downwards under the substructure relation.

\begin{definition}
The classes $\Av(A)$ and $\Av(B)$ are \emph{Wilf-equivalent}, written $\Av(A) \we \Av(B)$, if there is a bijection between them that preserves the size of each structure. 
Equivalently, the generating functions $F_A$ of $\Av(A)$ and $F_B$ of $\Av(B)$ are equal.
\end{definition}

Sometimes \emph{par abus de langage} we may say that $A$ and $B$ are Wilf-equivalent when we mean that $\Av(A)$ and $\Av(B)$ are. 
If $A$ and $B$ are of different sizes, then they cannot possibly be Wilf-equivalent, 
so effectively Wilf-equivalence is an equivalence relation on structures of size $n$ for each $n$. 
As such, the $n^{\mbox{\scriptsize th}}$ Catalan number is an upper bound for the number of its equivalence classes there, but we shall see that this is far from the truth.

\section{Arch systems containing and avoiding subsystems}

If an arch system $X$ contains some arch system $P$ then there is a \emph{leftmost} occurrence of $P$ in $X$ (which we often denote $P_L$) 
by which we mean the occurrence of $P$ whose rightmost point (i.e.~the point of $X$ that corresponds to the final point of $P$ in this occurrence) is as far left as possible. 
If there are two such occurrences with the same rightmost point, we designate as $P_L$ the one whose second rightmost point is as far left as possible etc. 
There is also a corresponding notion of \emph{rightmost} occurrence.

One advantage of working with arch systems is that it is clear that, 
when searching for a substructure of $X$ equal to some given arch system we may proceed in a greedy fashion. That is:

\begin{observation}
\label{ob:greedyInvolvement}
Suppose that $P$, $Q$ and $X$ are arch systems and that $PQ$ is a substructure of $X$. Then, in witnessing this we may use the leftmost occurrence, $P_L$, of $P$ in $X$.
\end{observation}

We will use this observation (and some obvious generalisations) repeatedly without further comment. 
Note however that we do not suggest that $X$ must factor into a part containing $P$ and a part containing $Q$. 
For example the system $\archOver{P \emptyArch}\emptyArch$ has $P \emptyArch \emptyArch$ as a substructure, but no such factorisation.

\medskip

For any arch system $A$, let $F_A$ denote the generating function of $\Av(A)$. 
It is a result of \cite{Mansour:Restricted} (expressed in somewhat different terms of course) that $F_A$ is necessarily a rational function. 
In fact, given a factorisation of $A$ into atoms we can write down a system of equations that allow for the recursive computation of $F_A$ 
(again, this is already done in \cite{Mansour:Restricted} and, in somewhat more general terms, in \cite{AA05}). 
The following proposition simply translates that result into the current context.

\begin{proposition}
\label{pr:classEnumeration}
Let $A$ be an arch system, with $A = a_1 a_2 \cdots a_m$ its factorisation into atoms, and $a_1 = \archOver{A_1}$. Then the generating function of $\Av(A)$ is 
\[
F_A = 1 + t F_{A_1} F_A + t (F_{a_1} - F_{A_1}) F_{a_2 \dots a_m} + t \sum_{k=2}^m (F_{a_1 \dots a_k} - F_{a_1 \dots a_{k-1}}) F_{a_k \dots a_m} .
\]
In particular, $F_A$ is rational. 
\end{proposition}

Fundamentally the first part of the proposition is proved simply by partitioning $A$-avoiding arch systems according to ``how much of $A$'' can be found within the first arch, and the conclusion of the second part follows by an easy inductive argument.

\section{A refinement of Wilf-equivalence}
\label{sec:equivalenceRelation}

In this section, we introduce an equivalence relation, $\sim$, on the collection of arch systems. 
We will then establish that this relation refines Wilf-equivalence, i.e.~that $A \sim B$ implies $\Av(A) \we \Av(B)$. So, without further ado:

\begin{definition}
\label{de:sim}
The binary relation, $\sim$, on arch systems is the finest equivalence relation that satisfies:
\begin{align}
\label{eq:archOverEquivalence}
 & A \sim B \implies \archOver{A} \sim \archOver{B} \\
\label{eq:substituteAtoms}
 & a \sim b \implies PaQ \sim PbQ \\
\label{eq:commuteAtoms}
 & PabQ \sim PbaQ \\
\label{eq:binaryRotation}
 & a \archOver{bc} \sim \archOver{ab} c .
\end{align}
where $A$, $B$, $P$ and $Q$ denote arbitrary arch systems; and $a$, $b$ and $c$ denote arbitrary atoms or empty arch systems. 
The equivalence classes of $\sim$ will be called \emph{cohorts}.
\end{definition}

Note that if $A \sim B$ then $A$ and $B$ have the same number of arches.
Note also that $A \sim B \Leftrightarrow \archOver{A} \sim \archOver{B}$, since (non trivial) equivalences between atoms may only be produced by rule (\ref{eq:archOverEquivalence}).

The main result which we prove in the following subsections is

\begin{theorem}
\label{th:mainTheorem}
If $A$ and $B$ are arch systems and $A \sim B$ then $\Av(A) \we \Av(B)$.
\end{theorem}

Interestingly, another equivalence relation (say, $\equiv$) on Catalan structures has been defined in a similar fashion by Rudolph~\cite{Rudolph:Pattern}. 
She proves in this paper that two $\equiv$-equivalent 132-avoiding permutations $\pi$ and $\tau$ are \emph{equipopular}, 
that is: for any $n$, the total number of occurrences of $\pi$ and $\tau$ in 132-avoiding permutations of size $n$ are equal. 
In other words, $\equiv$ refines equipopularity, and the analogy with $\sim$ refining Wilf-equivalence is clear. 
What it further interesting in the case of $\equiv$, is that it \emph{coincides} with equipopularity, as shown in~\cite{Chua:Equipopularity}. 
As a consequence, the number of equivalence classes for equipopularity among permutations of size $n$ is given by the number of partitions of $n$.

We separate the proof of Theorem~\ref{th:mainTheorem} into bijective and analytic proofs -- including some bijective proofs for cases where analytic ones are available. 
One reason for this is that the bijective proofs can frequently be refined to allow for term by term comparisons between the generating functions for inequivalent cohorts, 
while this is not so easily accomplished when only analytic proofs are available. 
A second reason is that these bijective proofs are needed for proving our claim of the introduction: 
that we are able (at least in principle) to provide bijections between any two classes of permutations $\Av(231,\pi)$ 
and $\Av(231,\tau)$ for $\pi$ and $\tau$ of size $n$ whose generating function is $C_n$.

To prove Theorem \ref{th:mainTheorem} it is sufficient to show that its conclusion holds for each of the four cases arising in Definition \ref{de:sim}. The proof is therefore subdivided into such cases. 
For compactness of notation we have found it convenient to denote functional application in exponential form, 
i.e.~the image of an arch system $X$ under a map $\tau$ will be denoted $X^\tau$.

\subsection{Bijective proofs}
\label{subsec:bijectiveProofs}

\begin{proof}[Proof of case (\ref{eq:archOverEquivalence})]
Let $A$ and $B$ be given with $A \sim B$, and suppose that $\Av(A) \we \Av(B)$. 
We may further assume that $A$ and $B$ are not empty, or the result trivially holds. 
Take $\sigma$ to be any size-preserving bijection between $\Av(A)$ and $\Av(B)$. 
Define a map $\tau$ on atoms $x = \archOver{X}$ belonging to $\Av(\archOver{A})$ by $x^\tau = \archOver{X^\sigma}$. 
This is possible since $x \in \Av(\archOver{A})$ if and only if $X \in \Av(A)$. 
Now extend $\tau$ to concatenations of atoms in the obvious way, i.e.~$(x_1 x_2 \dots x_m)^\tau = x_1^\tau x_2^\tau \dots x_m^\tau$. 
Since $\Av(\archOver{A})$ consists exactly of arch systems which are concatenations of atoms whose contents belong to $\Av(A)$ 
(and correspondingly $\Av(\archOver{B})$ consists exactly of arch systems which are concatenations of atoms whose contents belong to $\Av(B)$), 
$\tau : \Av(\archOver{A}) \to \Av(\archOver{B})$ is a size preserving bijection.
\end{proof}

\begin{proof}[Proof of case (\ref{eq:substituteAtoms})]
Let arbitrary arch systems $P$ and $Q$ and atoms $a$ and $b$ be given with $a \sim b$. 
Assume that $a$ and $b$ are not empty (or the result trivially holds),  
and let $\sigma: \Av(a) \to \Av(b)$ be a size preserving bijection. We will define a size preserving bijection $\tau: \Av(PaQ) \to \Av(PbQ)$.

Suppose that $X \in \Av(PaQ)$. If $X \in \Av(PQ)$ we define $X^\tau = X$. 
Otherwise take the leftmost copy, $P_L$, of $P$ in $X$ and the rightmost copy, $Q_R$, of $Q$. 
The arches that begin before the end of $P_L$ but end after it, and those that end after the beginning of $Q_R$ but begin before it 
divide the segment between the end of $P_L$ and the beginning of $Q_R$ into intervals. 
This is illustrated in Figure \ref{fi:substituteAtoms}. Since $a$ is an atom, any occurrence of $a$ between the end of $P_L$ and the beginning of $Q_R$ would have to be entirely contained in one of the intervals.
So, each of these intervals contains an arch system that avoids $a$ and conversely, if we are given an arch system with this property, it avoids $PaQ$. 
So define $X^\tau$ by applying $\sigma$ to each of the intervals while retaining the structure of $X$ up to the end of $P_L$ and from the beginning of $Q_R$ (including the arches that define the intervals). 
It is immediate to check that this defines a bijection from $\Av(PaQ)$ to $\Av(PbQ)$. 
\end{proof}

\begin{figure}
\centerline{%
\begin{tikzpicture}[scale=0.5]
\draw (0,0) -- (20,0);
\draw[dashed] (5,-1) -- (5,5);
\draw[dashed] (15,-1) -- (15,5);
\draw (4.8,0)  arc (0:45:10 and 5);
\draw (4,1) node {$P_L$};
\draw[dashed] (8,0) arc (0:90:6 and 4);
\draw[dashed] (10,0) arc (0:90:8 and 5);
\draw (15.2,0) arc (180:135:10 and 5);
\draw (16,1) node {$Q_R$};
\draw[dashed] (12,0) arc (180:90:6 and 5);
\draw (6.5,-1) node {$I_1$};
\draw (9,-1) node {$I_2$};
\draw (11,-1) node {$I_3$};
\draw (13.5,-1) node {$I_4$};
\end{tikzpicture}
}
\caption{The situation arising in the proof of case (\ref{eq:substituteAtoms}). 
In an arch system $X$ involving $PQ$ but avoiding $PaQ$ the leftmost copy of $P$, denoted $P_L$, and the rightmost copy of $Q$, denoted $Q_R$ are designated. 
Arches with one endpoint inside and one endpoint outside the interval between $P_L$ and $Q_R$ create a sequence of subintervals ($I_1$ through $I_4$ here) that must avoid $a$. 
To produce a $PbQ$ avoiding arch system, a bijection mapping $a$-avoiding systems to $b$-avoiding systems is applied to the $I_i$ and the remainder of the system is left unchanged.}
\label{fi:substituteAtoms}
\end{figure}
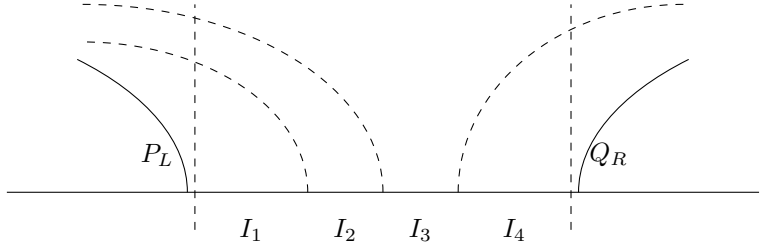

\begin{proof}[Proof of case (\ref{eq:commuteAtoms})]
The claim is trivial when $a$ or $b$ is empty. For the non-trivial case let $a$ and $b$ be non empty arbitrary atoms and $P$ and $Q$ arbitrary arch systems. We wish to construct a bijection $\tau: \Av(P a b Q) \to \Av(P b a Q)$. It will be helpful in what follows for the reader to refer to Figure \ref{fi:commuteAtoms}.
As in the previous case consider an arch system $X$. If $X$ avoids $PaQ$ then define $X^\tau = X$. 
Otherwise take $P_L$ to be the leftmost $P$, $a_L$ the leftmost atom involving $a$ following $P_L$ and $Q_R$ the rightmost $Q$ in $X$. 
Furthermore, denote by $C$ the contents of $a_L$. 
As in the previous proof the interval between $P_L$ and $Q_R$ is subdivided by those arches that have only one endpoint in this interval, 
say there are $i$ (resp. $j$) such arches with only their right (resp. left) endpoint between $P_L$ and $Q_R$. 
But now also one of those intervals (the one containing $a_L$) is further subdivided before and after $a_L$ by $a_L$ itself and any arches nested over $a_L$. 
Denote by $k$ the number of such arches (including the outermost arch of $a_L$). 
All the designated subintervals to the left of $a_L$ must avoid $a$ (since $a_L$ was leftmost) while those to the right of it must avoid $b$ (since $X$ avoids $PabQ$). 
To define $X^\tau$ simply reverse the order of these subintervals (keeping the arch systems within them fixed i.e.
~the contents of a subinterval are not changed, only its position between $P_L$ and $Q_R$). 
The structure of the arch system outside these intervals is unchanged, that is: 
the arch system before $P_L$ and after $Q_R$ is not modified, and there are still $k$ arches on top of $C$, 
and $i$ (resp. $j$) arches with only their right (resp. left) endpoint between $P_L$ and $Q_R$. 
In the resulting arch system $X^\tau$, $P_L$ and $Q_R$ are still the leftmost copies of $P$ and the rightmost copies of $Q$ respectively 
(since nothing before the end of $P_L$ or after the start of $Q_R$ has been changed). 
Between these, the atom $a_L$ has become the rightmost atom involving $a$. 
Since all of the intervals before it but following $P_L$ avoid $b$, $X^\tau$ avoids $PbaQ$. 
Moreover, it is clear that we can reverse this construction, so $\tau: \Av(PabQ) \to \Av(PbaQ)$ is a size preserving bijection as claimed.
\end{proof}

Remark that in the proof of case (\ref{eq:commuteAtoms}), we have chosen to reverse $A_1 \dots A_{i+k} C B_1 \dots B_{j+k}$ to $B_{j+k} \dots B_1 C A_{i+k} \dots A_1$ in $X^{\tau}$. 
But many variants of $\tau$ could have been defined by choosing any other permutation of the $A_{\ell}$, $B_m$ and $C$ 
that respects that all the $B_m$ are to the left of $C$ and all the $A_{\ell}$ to its right. 

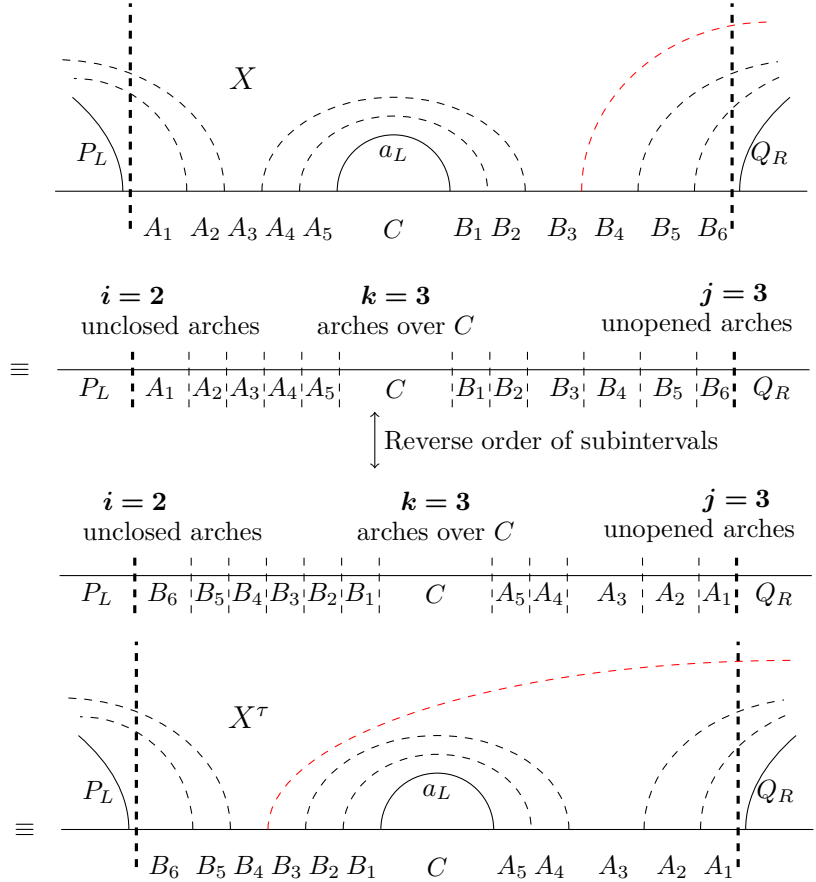
\begin{figure}
\centerline{%
\begin{tikzpicture}[scale=0.5]
\draw (-1.4,0) node {~};
\draw (0,0) -- (20,0);
\draw[dashed, very thick] (2,-1) -- (2,5);
\draw[dashed, very thick] (18,-1) -- (18,5);
\draw (1.8,0)  arc (0:30:10 and 5);
\draw (1,1) node {$P_L$};
\draw[dashed] (3.5,0) arc (0:90:3 and 3);
\draw[dashed] (4.5,0) arc (0:90:4.5 and 3.5);
\draw (18.2,0) arc (180:150:10 and 5);
\draw (5.5, 0) [dashed] arc(180:0:3.5 and 2.5);
\draw (6.5, 0) [dashed] arc(180:0:2.5 and 2);
\draw (7.5, 0) arc (180:0:1.5 and 1.5);
\draw (9., 1) node {$a_L$};
\draw (19,1) node {$Q_R$};
\draw[red, dashed] (14,0) arc (180:90:5 and 4.5);
\draw[dashed] (15.5,0) arc (180:100:4.5 and 3.5);
\draw[dashed] (17,0) arc (180:120:4.5 and 3.5);
\draw (2.75,-1) node {$A_1$};
\draw (4,-1) node {$A_2$};
\draw (5,-1) node {$A_3$};
\draw (6,-1) node {$A_4$};
\draw (7,-1) node {$A_5$};
\draw (9,-1) node {$C$};
\draw (11,-1) node {$B_1$};
\draw (12,-1) node {$B_2$};
\draw (13.5,-1) node {$B_3$};
\draw (14.75,-1) node {$B_4$};
\draw (16.25,-1) node {$B_5$};
\draw (17.5,-1) node {$B_6$};
\draw (5,3) node {\large $X$};
\end{tikzpicture}
}
\vspace{2ex}
\centerline{
\begin{tikzpicture}[scale=0.5]
\draw (-1,0) node {$\equiv$};
\draw (0,0) -- (20,0);
\draw[dashed, very thick] (2,-1) -- (2,.5);
\draw[dashed, very thick] (18,-1) -- (18,.5);
\draw (2,2) node {$\bm{i=2}$};
\draw (18,2) node {$\bm{j=3}$};
\draw (9,2) node {$\bm{k=3}$};
\draw (3,1.2) node {unclosed arches};
\draw (17,1.2) node {unopened arches};
\draw (9,1.2) node {arches over $C$};
\draw (1,-0.5) node {$P_L$};
\draw[dashed] (3.5,-1) -- (3.5,.5);
\draw[dashed] (4.5,-1) -- (4.5,.5);
\draw[dashed] (5.5, -1) -- (5.5, .5);
\draw[dashed] (6.5, -1) -- (6.5, .5);
\draw[dashed] (7.5, -1) -- (7.5, .5);
\draw[dashed] (10.5,-1) -- (10.5,.5);
\draw[dashed] (11.5, -1) -- (11.5, .5);
\draw[dashed] (12.5, -1) -- (12.5, .5);
\draw (19,-0.5) node {$Q_R$};
\draw[dashed] (14,-1) -- (14,.5);
\draw[dashed] (15.5,-1) -- (15.5,.5);
\draw[dashed] (17,-1) -- (17,.5);
\draw (2.75,-0.5) node {$A_1$};
\draw (4,-0.5) node {$A_2$};
\draw (5,-0.5) node {$A_3$};
\draw (6,-0.5) node {$A_4$};
\draw (7,-0.5) node {$A_5$};
\draw (9,-0.5) node {$C$};
\draw (11,-0.5) node {$B_1$};
\draw (12,-0.5) node {$B_2$};
\draw (13.5,-0.5) node {$B_3$};
\draw (14.75,-0.5) node {$B_4$};
\draw (16.25,-0.5) node {$B_5$};
\draw (17.5,-0.5) node {$B_6$};
\end{tikzpicture}
}
\centerline{
\begin{tikzpicture}[scale=0.5]
\draw (-7.5,0) node {~};
\draw[<->] (0,0) -- (0,1.5); 
\draw[anchor = west] (0,0.75) node {Reverse order of subintervals};
\end{tikzpicture}
}
\centerline{
\begin{tikzpicture}[scale=0.5]
\draw (-1.4,0) node {~};
\draw (0,0) -- (20,0);
\draw[dashed, very thick] (2,-1) -- (2,.5);
\draw[dashed, very thick] (18,-1) -- (18,.5);
\draw (2,2) node {$\bm{i=2}$};
\draw (18,2) node {$\bm{j=3}$};
\draw (10,2) node {$\bm{k=3}$};
\draw (3,1.2) node {unclosed arches};
\draw (17,1.2) node {unopened arches};
\draw (10,1.2) node {arches over $C$};
\draw (1,-0.5) node {$P_L$};
\draw[dashed] (3.5,-1) -- (3.5,.5);
\draw[dashed] (4.5,-1) -- (4.5,.5);
\draw[dashed] (5.5, -1) -- (5.5, .5);
\draw[dashed] (6.5, -1) -- (6.5, .5);
\draw[dashed] (7.5, -1) -- (7.5, .5);
\draw[dashed] (8.5, -1) -- (8.5, .5);
\draw[dashed] (11.5, -1) -- (11.5, .5);
\draw[dashed] (12.5, -1) -- (12.5, .5);
\draw[dashed] (13.5, -1) -- (13.5, .5);
\draw (19,-0.5) node {$Q_R$};
\draw[dashed] (15.5,-1) -- (15.5,.5);
\draw[dashed] (17,-1) -- (17,.5);
\draw (2.75,-0.5) node {$B_6$};
\draw (4,-0.5) node {$B_5$};
\draw (5,-0.5) node {$B_4$};
\draw (6,-0.5) node {$B_3$};
\draw (7,-0.5) node {$B_2$};
\draw (8,-0.5) node {$B_1$};
\draw (10,-0.5) node {$C$};
\draw (12,-0.5) node {$A_5$};
\draw (13,-0.5) node {$A_4$};
\draw (14.7,-0.5) node {$A_3$};
\draw (16.25,-0.5) node {$A_2$};
\draw (17.5,-0.5) node {$A_1$};
\end{tikzpicture}
}
\vspace{2ex}
\centerline{%
\begin{tikzpicture}[scale=0.5]
\draw (-1.2,0) node {\ \ $\equiv$};
\draw (0,0) -- (20,0);
\draw[dashed, very thick] (2,-1) -- (2,5);
\draw[dashed, very thick] (18,-1) -- (18,5);
\draw (1.8,0)  arc (0:30:10 and 5);
\draw (1,1) node {$P_L$};
\draw[dashed] (3.5,0) arc (0:90:3 and 3);
\draw[dashed] (4.5,0) arc (0:90:4.5 and 3.5);
\draw (18.2,0) arc (180:150:10 and 5);
\draw (6.5, 0) [dashed] arc(180:0:3.5 and 2.5);
\draw (7.5, 0) [dashed] arc(180:0:2.5 and 2);
\draw (8.5, 0) arc (180:0:1.5 and 1.5);
\draw (10., 1) node {$a_L$};
\draw (19,1) node {$Q_R$};
\draw[red, dashed] (5.5,0) arc (180:90:14 and 4.5);
\draw[dashed] (15.5,0) arc (180:100:4.5 and 3.5);
\draw[dashed] (17,0) arc (180:120:4.5 and 3.5);
\draw (2.75,-1) node {$B_6$};
\draw (4,-1) node {$B_5$};
\draw (5,-1) node {$B_4$};
\draw (6,-1) node {$B_3$};
\draw (7,-1) node {$B_2$};
\draw (8,-1) node {$B_1$};
\draw (10,-1) node {$C$};
\draw (12,-1) node {$A_5$};
\draw (13,-1) node {$A_4$};
\draw (14.7,-1) node {$A_3$};
\draw (16.25,-1) node {$A_2$};
\draw (17.5,-1) node {$A_1$};
\draw (5,3) node {\large $X^\tau$};
\end{tikzpicture}
}
\caption{The situation arising in the proof of case (\ref{eq:commuteAtoms}). 
In the top diagram the original $PabQ$ avoiding arch system is shown. 
Each interval $A_i$ must avoid $a$ and each interval $B_j$ must avoid $b$. 
In the bottom diagram its image is shown -- 
the atom $a_L$ and the nest of arches around it are moved to the right to allow copies of the $B_j$ to be placed on the left, 
and copies of $A_i$ on the right, as seen in the middle two diagrams.}
\label{fi:commuteAtoms}
\end{figure}

Turning now to case (\ref{eq:binaryRotation}), we will give an analytic proof below, but here give a bijective proof of a special case of it (which we will make use of later). 
Namely, we prove that $\Av( a \archOver{b} ) \we \Av(\archOver{ba})$, 
which with cases (\ref{eq:archOverEquivalence}) and (\ref{eq:commuteAtoms}), is equivalent to case (\ref{eq:binaryRotation}) with (at least) one of $a$, $b$ and $c$ empty. 

\begin{proof}[Bijective proof of specialisation of case (\ref{eq:binaryRotation}): $\Av( a \archOver{b} ) \we \Av(\archOver{ba})$]
We may assume that $a$ is not empty (otherwise there is nothing to prove). 
We will also assume that $b$ is not empty, but will indicate along the proof how it can be modified in case $b$ is empty. 
The proof goes along familiar lines, so we will be somewhat brief. 
Let $X \in \Av(a \archOver{b} )$ be given. We wish to define its image $X^\tau$, and will assume that $Y^\tau$ has already been defined for all $Y$ of smaller size. 
If $X \in \Av(b)$ let $X^\tau = X$. Otherwise consider the rightmost occurrence, $b_R$, of $b$. 
Since $b$ is an atom, this occurrence ends with the final arch of something of the form $\archOver{C}$ where the contents of $b$ occur in $C$, but $b$ does not. 
Consider the intervals defined by the nest of arches (if any) over $\archOver{C}$. 
Immediately to the left of $\archOver{C}$, we have an interval $M$ and the only condition is that it must avoid $a \archOver{b}$. 
Once we move past the first enclosing arch to the left the remaining intervals (of which there are, say $p$ called $A_1$ through $A_p$) must avoid $a$. 
To the right of $\archOver{C}$ all the intervals (of which there are $p+1$, $B_0$ through $B_p$) must avoid $b$. So
\[
X = A_p \archOver{A_{p-1} \archOver{\ldots \archOver{A_1 \archOver{M \archOver{C} B_0} B_1} \ldots} B_{p-1}} B_p.
\]
Now set:
\[
X^{\tau} = B_0 \archOver{B_p \archOver{B_{p-1} \archOver{\ldots \archOver{B_{1} \archOver{C} A_1} \ldots} A_{p-1}} A_p} M^\tau.
\]
In the case where $b$ is empty, we should instead decompose $X$ according to its last arch as $X = A_1 \archOver{M}$, where $A_1$ avoids $a$ and $M$ avoids $a \emptyArch$, 
and set $X^{\tau} = \archOver{A_1} M^\tau$.

That $X^{\tau}$ avoids $\archOver{ba}$ follows by induction inside $M^\tau$ and because the $A_i$ (resp. $B_i$) all avoid $a$ (resp. $b$).

Finally, the decomposition of arch systems avoiding $\archOver{ba}$ according to their leftmost occurrence of $b$ (resp. their first arch is $b$ is empty) 
allows to describe them canonically as 
\[
B_0 \archOver{B_p \archOver{B_{p-1} \archOver{\ldots \archOver{B_{1} \archOver{C} A_1} \ldots} A_{p-1}} A_p} M' \text{\qquad (resp. } \archOver{A_1} M' \text{),}
\]
where each $B_i$ avoids $b$, $C$ avoids $b$ but involves the contents of $b$, each $A_i$ avoids $a$, and $M'$ avoids $\archOver{ba}$. 
So the above construction can be reverse, and $\tau: \Av(a\archOver{b}) \to \Av(\archOver{ba})$ is a size preserving bijection as claimed.
\end{proof}

Note that, as in the proof of case (\ref{eq:commuteAtoms}), we can again define many variants of the bijection $\tau: \Av(a\archOver{b}) \to \Av(\archOver{ba})$, 
by replacing in $X^\tau$ the sequence $B_0 B_p B_{p-1} \ldots B_1$ (resp. $A_1 \ldots A_{p-1} A_p$) by any permutation of the $B_i$ (resp. $A_i$).

\subsection{Analytic proofs}

To complete the proof of Theorem \ref{th:mainTheorem} we need to consider the full version of case (\ref{eq:binaryRotation}) 
i.e.~we must show that $\Av(a \archOver{bc}) \we \Av(\archOver{ab} c)$ when none of $a$, $b$ and $c$ is empty.

\begin{proof}[Proof of case (\ref{eq:binaryRotation})]
Let $a = \archOver{A}$, $b = \archOver{B}$ and $c = \archOver{C}$. For an arch system $X$ let $F_X$ be the generating function of $\Av(X)$. 
Using the general technique described in Proposition \ref{pr:classEnumeration} we can compute the generating function $F_{a\archOverSubscript{bc}}$ in terms of $F_A$, $F_B$ and $F_C$. 

\begin{align*}
F_{a\archOverSubscript{bc}} &= 1+ tF_{A}F_{a\archOverSubscript{bc}} + t(F_{a\archOverSubscript{bc}} - F_A) F_{\archOverSubscript{bc}} \\
F_{\archOverSubscript{bc}} &= 1 + t F_{bc} F_{\archOverSubscript{bc}} \\
F_{bc} &= 1+ tF_{B}F_{bc} + t(F_{bc} - F_B) F_{c}\\
F_c &= 1 + t F_C F_c
\end{align*}

Solving the system\footnote{Or rather, having \emph{Mathematica} solve it for you.} for $F_{a\archOverSubscript{bc}}$ in terms of $F_A$, $F_B$ and $F_C$ gives a terrible mess 
which is nevertheless symmetric in $F_A$, $F_B$ and $F_C$. 
In fact the solution is tidier if written in terms of $F_a$, $F_b$ and $F_c$ (recall that $F_a = 1/(1 - tF_A)$, i.e.~$F_A = (F_a - 1)/(t F_a)$ etc.): 
\[
F_{a\archOverSubscript{bc}} = \frac{1 - t (F_a F_b + F_b F_c + F_c F_a - F_a F_b F_c)}{1 - t (F_a + F_b + F_c - F_a F_b F_c) }
\]

Accordingly, $F_{a\archOverSubscript{bc}}$ is symmetric in $F_a$, $F_b$ and $F_c$. 
This proves that $\Av(a\archOver{bc}) \we \Av(c \archOver{ab})$. 
Now use case (\ref{eq:commuteAtoms}) to reach the desired conclusion.
\end{proof}

We have seen in the above proof that, for any atom $a=\archOver{A}$, $F_a$ completely determines $F_A$ and conversely, 
via the relations $F_a = 1/(1 - tF_A)$ and $F_A = (F_a - 1)/(t F_a)$. 
This simple fact also provides an analytic proof that:
\begin{observation}
For any atoms $a=\archOver{A}$ and $b=\archOver{B}$, $\Av(A) \we \Av(B)$ if and only if $\Av(a) \we \Av(b)$. 
\label{obs:GFforArchOver}
\end{observation}

\section{The combinatorial class of cohorts}
\label{sec:cohorts}

From Theorem~\ref{th:mainTheorem} it follows that the number of different generating functions of classes of arch systems avoiding an arch system with $n$ arches
(or equivalently, the number of Wilf-equivalence classes of permutation classes $\Av(231,\pi)$ for $\pi$ of size $n$ avoiding 231) 
is at most the number of cohorts (i.e.~equivalence classes of $\sim$) for $n$ element structures. 
In Conjecture~\ref{co:cohortDeterminesGF} we suggest that these numbers may actually be equal, explaining our interest in the enumeration of cohorts. In any case, the number of cohorts certainly provides an upper bound for the number of such Wilf-equivalence classes.
Towards the goal of enumerating cohorts, we first associate with each cohort a \emph{single} structure, and then enumerate such structures. 
These structures that represent cohorts may be seen as choosing one representative in the set of all structures (e.g.~all arch systems) that form a cohort. 
Alternatively -- and it is rather this point of view we choose -- 
we can think of the structure representing a cohort as an abstract structure from which all structures in the cohort may be recovered.

\subsection{The structure of a cohort}
\label{subsec:structureCohort}

It is easiest to describe the single (abstract) structure associated with a cohort in the context of plane forests. 
Note first that these structures representing cohorts should be non-plane objects. Indeed:

\begin{proposition}
If two plane forests $A$ and $B$ are isomorphic as non-plane forests, then $A \sim B$.
\end{proposition}

\begin{proof}
This follows directly by induction from rules (\ref{eq:archOverEquivalence}), (\ref{eq:substituteAtoms}) and (\ref{eq:commuteAtoms}). 
Specifically, suppose that plane forests $A$ and $B$ which are isomorphic as non-plane forests are given and that the result holds for all plane forests of lesser size. 
If $A$ and $B$ are trees (corresponding to atoms in the context of arch systems), 
then the result applies to the forests obtained by deleting their roots (i.e.~the contents of these atoms), and hence by rule (\ref{eq:archOverEquivalence}) to $A$ and $B$. 
Otherwise, each of $A$ and $B$ is the concatenation of the same number of trees (i.e.~atoms), say $m$. 
First, using rule (\ref{eq:commuteAtoms}) we can find $A' \sim A$ so that $A' = a_1 a_2 \dots a_m$, $B = b_1 b_2 \dots b_m$, and each tree $a_i$ is isomorphic $b_i$. 
Then using rule (\ref{eq:substituteAtoms}) we are done.
\end{proof}

We note that this proposition already establishes that there are no more cohorts for $n$ element structures 
than there are rooted non-plane forest with $n$ nodes, or equivalently rooted non-plane \emph{trees} with $n+1$ nodes. 
As the asymptotic enumeration of these (see for example~\cite{Flajolet:Analytic}*{Proposition VII.5 and note VII.21}) has exponential growth rate approximately $2.956$ 
we already see exponentially fewer Wilf-equivalence classes than there are structures of size $n$. 
However, the final rule provides a further reduction.

Let us focus our attention on $\sim$-equivalences \emph{between atoms (or trees) only} that may be derived from rule (\ref{eq:binaryRotation}). 
In this context, an equivalent form of this rule is $\archOver{a \archOver{bc}} \sim \archOver {\archOver{ab} c}$. 
So in terms of trees, rule (\ref{eq:binaryRotation}) allows us to rotate subtrees at binary branches. 
Furthermore, it also allows unary nodes to be lifted through binary ones (from the case when $c$ is empty) via $\archOver{a \archOver{b}} \sim \archOver{\archOver{ab}}$. 
Finally, in the case were $b$ and $c$ are empty, rule (\ref{eq:binaryRotation}) rewrites as $\archOver{a \emptyArch} \sim \archOver{\archOver{a}}$, 
allowing to transform a leaf hanging below a binary node $x$ into a unary node between $x$ and its other child. 
These operations on trees are shown in Figure~\ref{fi:binaryRuleOnTrees}.

\begin{figure}
\centerline{
\begin{tikzpicture}[shape=circle, scale = 0.5]
\node[fill,circle,minimum size=4, inner sep = 0] at (-0.4,0) {} 
  child{
    node[inner sep = 0] {$T_a$}
  }
  child{
     node[fill,circle,minimum size=4, inner sep = 0] {}
       child{
       	node[inner sep = 0] {$T_b$}
	}
	child{
       	node[inner sep = 0] {$T_c$}
	}
  };  
\draw (1.8,-1) node {$\sim$}; 
\node[fill,circle,minimum size=4, inner sep = 0] at (4,0) {} 
  child{
     node[fill,circle,minimum size=4, inner sep = 0] {}
       child{
       	node[inner sep = 0] {$T_a$}
	}
	child{
       	node[inner sep = 0] {$T_b$}
	}
  }
  child{
    node[inner sep = 0] {$T_c$}
  };
\draw (6,-1) node {;}; 
\node[fill,circle,minimum size=4, inner sep = 0] at (8,0) {} 
  child{
     node[fill,circle,minimum size=4, inner sep = 0] {}    
     child{
       node[inner sep = 0] {$T_a$}
     }
  }
  child{
     node[inner sep = 0] {$T_b$}       
  };  
\draw (9.8,-1) node {$\sim$}; 
\node[fill,circle,minimum size=4, inner sep = 0] at (11,0) {}
     child{
     node[fill,circle,minimum size=4, inner sep = 0] {}
       child{
       	node[inner sep = 0] {$T_a$}
	}
       child{
       	node[inner sep = 0] {$T_b$}
	}
  };
  \draw (12.2,-1) node {$\sim$}; 
\node[fill,circle,minimum size=4, inner sep = 0] at (14,0) {} 
  child{
     node[inner sep = 0] {$T_a$}
  }
  child{
     node[fill,circle,minimum size=4, inner sep = 0] {}    
     child{
       node[inner sep = 0] {$T_b$}
     }
  };
\draw (16.5,-1) node {;}; 
\node[fill,circle,minimum size=4, inner sep = 0] at (19,0) {} 
  child{
    node[inner sep = 0] {$T_a$}
  }
  child{
     node[fill,circle,minimum size=4, inner sep = 0] {}       
  };  
\draw (20.5,-1) node {$\sim$}; 
\node[fill,circle,minimum size=4, inner sep = 0] at (21.5,0) {}
     child{
     node[fill,circle,minimum size=4, inner sep = 0] {}
       child{
       	node[inner sep = 0] {$T_a$}
	}
  };
  \draw (22.5,-1) node {$\sim$}; 
\node[fill,circle,minimum size=4, inner sep = 0] at (24,0) {} 
  child{
     node[fill,circle,minimum size=4, inner sep = 0] {}
  }
  child{
    node[inner sep = 0] {$T_a$}
  };
\draw (1.8,-4.5) node {$(i)$};
\draw (11,-4.5) node {$(ii)$};
\draw (21.5,-4.5) node {$(iii)$};
\end{tikzpicture}
}
\caption{$\sim$-equivalences on trees that are derived from rule (\ref{eq:binaryRotation}).
}
\label{fi:binaryRuleOnTrees}
\end{figure}
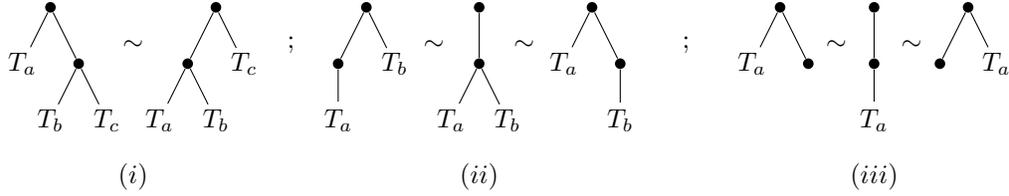

So, consider any subtree of a plane forest that has a binary root. 
In this tree replace any subtree whose root has three or more children by a symbol representing that atom (and temporarily call such atoms, \emph{large}). 
As a result we obtain a tree, $T$, all of whose internal nodes have one or two children and where the leaves are either large atoms, or bare nodes. 
As shown in Figure~\ref{fi:binaryRuleOnTrees}$(ii)$ and $(iii)$, we can lift the unary nodes and bare nodes through the binary ones 
to obtain a $\sim$-equivalent tree $T'$ with a chain of unary nodes running from the root, connected to a full binary tree all of whose leaves are labelled with large atoms.  
Finally, we can rotate the large atoms (see Figure~\ref{fi:binaryRuleOnTrees}$(i)$), permute them (from $PabQ \sim PbaQ$), 
and replace them by equivalent large atoms (from $a \sim b \Rightarrow PaQ \sim PbQ$). 
So we see that two such full binary trees (with leaves that are large atoms) are $\sim$-equivalent if and only if they have the same number of nodes (and hence leaves) 
and there is a bijection between their sets of leaves such that items in correspondence in these sets are $\sim$-equivalent large atoms. 
More properly, note that these ``sets'' of leaves are actually multisets, since repetitions are allowed. 

For ease of explanation, in the rest of this section we will focus on \emph{atomic cohorts}, i.e.~cohorts that contain at least one atom (or tree). 
Note that this is not an actual restriction: 
atomic cohorts for $(n+1)$ element structures are in bijective correspondence with cohorts for $n$ element structures, 
since $A \sim B \Leftrightarrow \archOver{A} \sim \archOver{B}$. 

The above discussion leads to a recursive description of (representatives for) atomic cohorts. 
Consider the recursive specification of a variety, $\A$ of non-plane tree-like structures:
\begin{equation}
\label{eq:recursiveCohortReps}
\begin{aligned}
 \A &=  \bullet \uplus (\bullet, \A)  \uplus \ \biguplus_{k\geq 2} (\bigtriangleup_{k-1}, \mset_k(\B)) \uplus \B \\
 \B &=  (\bullet, \mset_{\geq 3}(\A)),
\end{aligned}
\end{equation}
where $\bullet$ refers to a class with a single object of size $1$, 
parentheses denote ordered pairs, 
$\bigtriangleup_m$ denotes a class with a single object of size $m$, 
$\uplus$ denotes disjoint union, 
and $\mset$ denotes the multiset construction, with the subscript denoting the number of elements in the multiset. 
Equivalently, as non-plane trees: 
\begin{equation*}
 \A =  \bullet +
\tikz[scale=0.35,baseline=-0.3cm]{
 \node[fill,circle,minimum size=4, inner sep = 0] {}
       child{
       	node[inner sep = 0] {$\A$}
	};
}
+ \ \sum_{k\geq 2} 
\tikz[scale=0.35,baseline=-0.3cm]{
 \node[inner sep = 0] {~ $\bigtriangleup_{k-1}$}
       child{
       	node[inner sep = 0] {$\B$}
	}
       child{
        node[inner sep = 0] {$\ldots$}
        edge from parent [draw=none]}
       child{
       	node[inner sep = 0] {$\B$}
	};
\draw[decorate,decoration={brace}] (2.2,-2) -- (-2.2,-2) node[below,pos=0.5] {$k$ children};
}
+ \B \text{\quad and \quad}
\B =  \hspace*{-0.5cm} \tikz[scale=0.35,baseline=-0.3cm]{
 \node[fill,circle,minimum size=4, inner sep = 0] {}
       child{
       	node[inner sep = 0] {$\A$}
	}
       child{
        node[inner sep = 0] {$\ldots$}
        edge from parent [draw=none]}
       child{
       	node[inner sep = 0] {$\A$}
	};
\draw[decorate,decoration={brace}] (2.2,-2) -- (-2.2,-2) node[below,pos=0.5] {$\geq 3$ children};
}
\end{equation*}

\begin{proposition}
There is a size-preserving bijection between atomic cohorts and $\A$.
\end{proposition}

\begin{proof}
This is basically simply a direct translation of the preceding discussion, 
where we have unravelled all possible equivalences following from rules (\ref{eq:archOverEquivalence}) to (\ref{eq:binaryRotation}). 
The class $\B$ represents ``large atoms''. Then the elements of $\A$ are described in order as: 
a single node, a root with one child, an atom corresponding to a full binary tree with $k$ leaves labelled by large atoms, or a large atom.
\end{proof}

We shall use this description to refine the asymptotic enumeration of the number of cohorts.

\subsection{The number of cohorts}
\label{subsec:numberOfCohorts}

The first 15 values of the number of cohorts of arch systems of size $n \geq 1$ are given by:
\[ 
1, \ 1, \ 2, \ 4, \ 8, \ 16, \ 32, \ 67, \ 142, \ 307, \ 669, \ 1\,478, \ 3\,290, \ 7\,390, \ 16\,709, \ldots\]

Furthermore, for each cohort of size up to 15, we can produce a representative arch system $X$ for that cohort, 
and check that the generating functions of the classes $\Av(X)$ are all distinct. 
With Theorem~\ref{th:mainTheorem}, this ensures that the above also shows 
the first few terms of the sequence enumerating Wilf-equivalence classes of classes $\Av(A)$ for $A$ of size $n$. 
Notice that more terms of the enumeration sequence of cohorts may be obtained from Equation~(\ref{eq:cohortGF}) below -- 
namely, the next few terms are $38\,027, \ 86\,993, \ 200\,018, \ 461\,847, \ 1\,070\,675$. 
From Theorem~\ref{th:mainTheorem}, these are upper bounds on the number of Wilf-equivalence classes of $\Av(A)$, 
but we cannot ensure that they are equal (although we suspect they are). 
In the following, we therefore study the asymptotic behaviour of the number of cohorts of arch systems of size $n$. 
 
As already noted, the number of cohorts of arch systems of size $n$ equals the number of \emph{atomic} cohorts of arch systems of size $n+1$. 
Here we can make profitable use of (\ref{eq:recursiveCohortReps}) to provide a functional equation for 
the generating function $A(t) = \sum a_n t^n$ counting atomic cohorts which is susceptible to asymptotic analysis using the techniques of Section VII.5 of \cite{Flajolet:Analytic}, 
or with minor variations of \cite{Harary:Twenty}. Specifically we obtain:
\begin{equation}
\label{eq:cohortGF}
\begin{aligned}
A &= t + t A + \frac{1}{t} M_{\geq 2}(t B) + B \\
B &= t M_{\geq 3} (A) \\
\end{aligned}
\end{equation}
where 
\begin{align*}
M(Z) &= \exp(\frac{Z(t)}{1} + \frac{Z(t^2)}{2} + \frac{Z(t^3)}{3} + \frac{Z(t^4)}{4} + \ldots ) \\
M_{\geq 2} (Z) &= M(Z) - 1 - Z(t) \\
M_{\geq 3} (Z) &= M(Z) - 1 - Z(t) - \frac{1}{2} \left( Z(t)^2 + Z(t^2) \right) 
\end{align*}
are operators representing the generating functions that enumerate multisets of objects, 
and respectively such multisets of size at least 2 or 3 counted by the generating function $Z$.

Clearly the power series $A$ dominates $t + t M_{\geq 3}(A)$ term by term, 
and so $a_n$ is at least the number of non-plane trees with $n$ nodes in which each internal node has at least 3 children. 
This trivial estimate suffices to show that the radius of convergence, $\rho_A$, of $A$ is less than $1$ (and hence so is that of $B$). 
Now observe that in general
\begin{equation}
\label{eq:expFormOfM}
M(Z) = \exp(Z) \exp(W)
\end{equation}
where 
\[
W = \frac{Z(t^2)}{2} + \frac{Z(t^3)}{3} + \frac{Z(t^4)}{4} + \ldots
\]
If the radius of convergence of $Z$ is $r < 1$, then the radius of convergence of $W$ is easily seen to be at least $\sqrt{r} > r$. 
This suggests that when analysing the radius of convergence of generating functions defined by functional equations involving the $M$ operator, 
we treat these as implicit definitions of the desired function in terms of ``known'' analytic functions which, 
while related to the function we are analysing are analytic in a disc around the origin strictly containing the radius of convergence of the function we seek.
Effectively these are the first five steps of \cite{Harary:Twenty}. 
So to proceed we view (\ref{eq:cohortGF}) as an implicit definition of $A$ in terms of these ``known'' functions 
after having eliminated $B$ entirely and noting also that the terms corresponding to $Z(t^2)$ in any occurrences of $M_{\geq 3}$ should also be treated as ``known''. 
Thus we aim to find the radius of convergence of the solution to $F(t,y) = 0$ where:
\[
F(t, y) = -y + t + t y + t M_{\geq 3} (y) + M_{\geq 2}( t^2 M_{\geq 3} (y))/t
\]
In this expression we replace the subscripted $M$ operators by their definitions above, 
and then on the remaining occurrences of $M$ use the form given by \ref{eq:expFormOfM} 
to replace the definition of $F$ by one involving $y$, $t$ and some functions of $t$ known to be analytic on the domain of interest. 
Continuing with the steps of \cite{Harary:Twenty} as we know already that the solution $y$ is a generating function 
we can find its radius of convergence $\rho_A$ by determining the smallest positive root of the equation $F_y (t, y) = 0$ (where $F_y$ is the derivative of $F$ with respect to $y$).

Of course in finding this root we first take the derivative formally 
and then replace $y$ and all the related ``known'' functions by polynomial approximations of some degree, denoted $n$, 
obtained by using equation (\ref{eq:cohortGF}) as a recurrence for generating terms of $A$. 
The results of these approximations for various values of $n$ are as follows:
\[
\begin{array}{ccc}
n & \rho_A & 1/\rho_A \\ \hline
50 & 0.4069 & 2.4575 \\
100 & 0.4083 & 2.4763 \\
200 & 0.4022 & 2.4863 \\
400 & 0.4014 & 2.4916 \\
800 & 0.4009 & 2.4943
\end{array}
\]
These values agree well with the numerical estimates obtained by simply looking at computed coefficients of $A$ 
and fitting an asymptotic expression of the form $a_n \sim c n^{-3/2} \gamma^n$. 
Note however that the apparent accuracy is significantly less than that given in examples VII.21 and VII.22 of \cite{Flajolet:Analytic}. 
We suspect that this arises due to the iterated application of $M$ and the correction terms that are part of the definitions of $M_{\geq 2}$ and $M_{\geq 3}$. 
Another possible reason is that we also truncate the ``known'' parts at degree $n$. 
Approximate values of $\alpha$ and $\gamma$ are $\alpha \approx 0.454$ and $\gamma \approx 2.4975$. 

To justify the asymptotic form used above, thereby reaching step 14 of the 20 steps (which is where we intend to stop) 
requires checking that $F_{yy} (\rho_A, A(\rho_A)) \neq 0$. Fortunately, we can compute $F_{yy}$ modulo some ``known'' functions (in the usual sense) as:
\begin{align*}
F_{yy} &= e^{p(y)} t  \left( t^2 ( 1 + y - e^y c_1(t))^2 + e^y c_1(t) - 1 \right), \: \mbox{where} \\
p(y) &= (-1/2) t^2 ( 2 + 2 y + y^2 + c_2(t) - 2 e^{y} c_3(t) ) + c_4(t),
\end{align*}
with $c_2$, $c_3$ and $c_4$ analytic and real at $\rho_A$. 
Further $c_1(t) = \exp{  w(t) }$ where $w(t)$ is a series with positive coefficients. 
So $e^y c_1 - 1 > 0$ at $t = \rho_A$ and thus $F_{yy}(\rho_A, A(\rho_A)) \neq 0$.

Recall that atomic cohorts of arch systems with $n+1$ arches are in bijection with cohorts of arch systems 
with $n$ arches, so to obtain the general asymptotics we multiply the constant term  from the atomic asymptotics by $\gamma$ yielding: 
\begin{theorem}
\label{THM:enumeration_of_cohorts}
The number of cohorts of arch systems with $n$ arches behaves asymptotically as $c n^{-3/2} \gamma^n$, 
where $c \approx 1.13$ and $\gamma \approx 2.4975$. 
\end{theorem}

\section{The main cohort, and comparison between cohorts}

We start this section by defining a special cohort of arch systems of any size $n$ and studying its properties. 
We specifically deal with the number of arch systems contained in this cohort, 
and with the generating function of any class $\Av(X)$ for an arch system $X$ in this cohort.
This will complete the proofs of our claims of the introduction. 
This special cohort is called the main cohort, because it appears to be the largest with respect to two criteria. 

Accordingly, we report in this section some results about the comparison between cohorts (of structures of the same size, $n$) with respect to these two criteria. 
One is the size of these cohorts, i.e.~the number of equivalent arch systems they contain. 
Here, we focus on extremal cases: 
we conjecture that the main cohort is the one with maximal size, 
and we describe \emph{singleton cohorts}, that is: cohorts which contain one single arch system.
Cohorts may also be compared with respect to the (common) generating functions of the classes $\Av(X)$ they represent. 
We provide some rules on arch systems that allow the comparison between the generating functions of their cohorts, 
and show that the main cohort is largest in the sense that its generating function dominates that of any other cohort. 

\subsection{The main cohort}

Following the discussion of Subsection~\ref{subsec:structureCohort}, 
for each $n$ there is a unique cohort of structures of size $n$ that arises from all unary-binary plane forests (i.e.~no large atoms are involved)
-- by definition, such forests consist of at most two trees, which are themselves unary-binary trees. 
We call this the \emph{main cohort} for structures of size $n$ and denote it by $\bigclass_n$. 
A representative of this cohort is the system $N_n$ of $n$ nested arches, whose corresponding forest is a chain of $n$ nodes. 
But from its description in terms of forests, it is clear that the main cohort also includes 
all the arch systems of size $n$ that can be built using the following operations, and only these: 
concatenate two atoms that belong to $\bigclass_j$ and $\bigclass_k$ for $j+k=n$, 
or place an arch over an arch system of $\bigclass_{n-1}$. 
For the same reason, if we let $M_n$ denote the number of atoms (i.e.~trees) of size $n$ in the cohort $\bigclass_n$, 
it is immediate that the generating function $M(t) = \sum M_n t^n$ satisfies:
\[
M = t + t M + t M^2.
\]
This identifies $(M_n)$ as the sequence of Motzkin numbers (offset by 1): 
\[
M_{n+1} = \textrm{Motz}_n = \sum\limits_{k=0}^{\lfloor n/2\rfloor} {n \choose {2k}} \textrm{Cat}_k.
\]

Recalling that the number of atoms in the main cohort for structures of size $n+1$ is equal to 
the total number of arch systems in the main cohort for structures of size $n$, we obtain:
\begin{proposition}
The size of the main cohort for structures of size $n$ is the $n$-th Motzkin number: $|\bigclass_n| = \textrm{\emph{Motz}}_n$.
\label{pr:MotzkinManyInMainCohort}
\end{proposition}

Furthermore, to $\bigclass_n$ corresponds one generating function: that of any $\Av(X)$ for $X \in \bigclass_n$. 
Taking $X = N_n$, where $N_n$ is the nest of $n$ arches, these generating functions $C_n$ are easily seen to satisfy
\[
C_n = \frac{1}{1 - t C_{n-1}} \text{\quad  (with initial condition } C_0 = 1\text{),}
\]
giving that:
\begin{proposition}
\label{pr:GFofMainCohort}
For any structure $X$ in $\bigclass_n$, the generating function of $\Av(X)$ is $C_n$. 
\end{proposition}

This justifies the remarks concerning the sequence of generating functions $(C_n)$ made in the introduction. 

Note that Proposition~\ref{pr:GFofMainCohort} provides an alternative proof of the enumeration of $\Av(231,\pi)$ (by $C_{n}$ for $n=|\pi|$) for several families of patterns $\pi$
that appear in the literature: namely decreasing patterns and patterns of the form $1 n(n-1)\dots 3 2$~\cite{ChW}, 
reverse of 2-layered permutations and 132-avoiding wedge-patterns of~\cites{Mansour:Restricted,Mansour:Chebyshev}, 
and patterns $\lambda_k \oplus \lambda_{n-k}$ of~\cite{FPSAC2013}.
Indeed, all such patterns belong to the main cohort of the corresponding size. 

For any structure $A$ in $\bigclass_n$, it is easy to see that 
there exists a chain of $\sim$-equivalences from $A$ to $N_n$ that never uses rule~(\ref{eq:binaryRotation}) with all of $a$, $b$ and $c$ not empty. 
So the same holds for any pair of structures $A$ and $B$ in $\bigclass_n$. 
Therefore, the bijective proofs of Subsection~\ref{subsec:bijectiveProofs} provide, 
for any such pair, a bijection between $\Av(A)$ and $\Av(B)$. 
A special case of this statement answers a question raised in~\cite{Mansour:Chebyshev}, about the description of a bijection between 
$\Av(132,\pi)$ and $\Av(132,\tau)$, for $\pi$ any 2-layered pattern and $\tau$ any 132-avoiding wedge-pattern. 

The name \emph{main} cohort has been chosen because we suspect that this cohort is the largest in two senses. 
We shall see in Subsection~\ref{subsec:comparingGFofCohorts} that 
$C_n$ dominates (term by term) the generating function $F_X$ of $\Av(X)$
for any arch system $X$ of size $n$. 
Moreover, unless $X \in \bigclass_n$, eventually $C_n$ dominates $F_X$ strictly. 

Since the main cohort is constructed using the smallest building blocks 
i.e.~any other cohort must involve somehow one or more atoms consisting of at least four arches (such as $\archOver{\emptyArch \emptyArch \emptyArch}$) 
it seems natural to suspect that among the cohorts of $n$-arch systems, the main cohort is largest. 
Turning this intuition into a proof is however far from immediate, and we offer the following conjecture: 

\begin{conjecture}
\label{conj:mainCohortIsBiggest}
For every positive integer $n \geq 3$ the size of $\bigclass_n$ is greater than the size of any other cohort of an arch system of size $n$.
\end{conjecture}


\subsection{Singleton cohorts}

At the other end of the chain, it is amusing to consider the cohorts that contain only a single arch system. 
Modulo Conjecture \ref{co:cohortDeterminesGF} these correspond to the only arch systems, $A$, that can be recognised directly from the generating function of $\Av(A)$. 

\begin{proposition}
\label{pr:characterOfSingletonCohorts}
The cohort of a (non empty) arch system $A$ is a singleton if and only if:
\begin{itemize}
\item
$A = b^k$ where $k \geq 3$ and $b$ is an atom which is the only atom in its cohort\footnote{
Note that this condition is less restrictive than the cohort of $b$ being a singleton.}, or 
\item
$A = a^2$ where $a$ is an atom whose contents are some $b^k$ as in the first condition, or
\item
$A$ is an atom whose contents are either empty or some $b^k$ as in the first condition. 
\end{itemize}
Moreover, the atoms which are the only atoms in their cohort are:
$\emptyArch$ and the atoms whose (non empty) contents belong to a singleton cohort. 
\end{proposition}

\begin{proof}
Suppose first that an arch system $A$ is a concatenation of two or more atoms. 
For such arch systems rule (\ref{eq:commuteAtoms}) would yield more than one element in $A$'s cohort unless these atoms were all identical. 
Further, rule (\ref{eq:substituteAtoms}) would do likewise if that atom were not the only atom in its cohort. 

On the other hand, if these conditions are met, and $A$ is a concatenation of at least three atoms 
then rules (\ref{eq:archOverEquivalence}) and (\ref{eq:binaryRotation}) cannot be applied, 
so such $A$ are indeed arch systems whose cohort is a singleton.

If the cohort of $A = a^2$ is a singleton, and $a = \archOver{X}$ then clearly the cohort of $X$ must be a singleton (else rule (\ref{eq:archOverEquivalence}) would apply). 
Furthermore, $X$ must be the concatenation of at least three atoms, or else rule (\ref{eq:binaryRotation}) could be applied in $A$. 
Conversely, if $X$ satisfies these conditions then none of the rules can be applied to yield any other element of $A$'s cohort.

If $A = \archOver{X}$ is an atom that forms a singleton cohort, then its contents $X$ (if not empty) must belong to a singleton cohort 
(else rule (\ref{eq:archOverEquivalence}) would apply). 
$X$ cannot be an atom since $\archOver{\archOver{Y}} \sim \emptyArch{\archOver{Y}}$ (from rule (\ref{eq:binaryRotation}) with $c=\archOver{Y}$ and $a$ and $b$ empty). 
Similarly, $X$ cannot be the concatenation of two atoms, since $\archOver{ab} = a \archOver{b}$ (from rule (\ref{eq:binaryRotation}) with $c$ empty). 
So $X$ must satisfy the first condition. 
Conversely if the contents $X$ of $A$ do satisfy this condition then the cohort of $A$ will be a singleton: 
indeed, the only rules allowing one to find a $\sim$-equivalent of an atom are rule (\ref{eq:archOverEquivalence}) and 
the special cases of rule (\ref{eq:binaryRotation}) -- which do not apply here since $X$ is the concatenation of at least three atoms. 

If an atom in the only atom is its cohort, then obviously its contents are either empty or belong to a singleton cohort. 
Conversely, consider an atom that is either $\emptyArch$ or $\archOver{X}$ where the cohort of $X$ is a singleton. 
Certainly, $\emptyArch$ is the only atom in its cohort (which is indeed a singleton here). 
We claim that for any arch system $X$ whose cohort is a singleton, $\archOver{X}$ is the only atom in its cohort. 
Such $X$ satisfies one of the conditions of Proposition~\ref{pr:characterOfSingletonCohorts}. 
If $X=b^k$ as in the first condition, then none of the rules (\ref{eq:archOverEquivalence}) to (\ref{eq:binaryRotation}) apply to $\archOver{X}$
-- note that here the cohort of $\archOver{X}$ is actually a singleton, from the third condition. 
If $X=a^2$ as in the second condition, then only special cases of rule (\ref{eq:binaryRotation}) apply to $\archOver{X} = \archOver{aa}$, 
producing two $\sim$-equivalent to $\archOver{X}$, namely $\archOver{a}a$ and $a\archOver{a}$.
If $X= \archOver{Y}$ is an atom as in the third condition, then only special cases of rule (\ref{eq:binaryRotation}) apply to $\archOver{X} = \archOver{\archOver{Y}}$
producing two (one if $Y$ is empty) $\sim$-equivalent to $\archOver{X}$, namely $\archOver{Y} \emptyArch$ and $\emptyArch \archOver{Y}$.
In all cases, we observe that $\archOver{X}$ is indeed the only atom in its cohort.
\end{proof}

In order to translate these conditions into recurrences allowing to count singleton cohorts we introduce several auxiliary functions: 
$S_1(n)$ counts the atomic singleton cohorts, 
$S_2(n)$ counts the singleton cohorts of the form $a^2$, 
and $S_{\geq 3}(n)$ counts the singleton cohorts of the form $b^k$ for $k \geq 3$. 
Also $A(n)$ counts the number of cohorts that contain a single atom. 
Then we obtain as recursive conditions:
\begin{align*}
S_1(n) &= S_{\geq 3}(n-1), \\
S_2(n) &= \left\{ \begin{array}{ll} 0 & \mbox{$n$ odd} \\ S_{\geq 3}(n/2 - 1) & \mbox{$n$ even} \end{array} \right. , \\
S_{\geq 3}(n) &= \sum_{k \geq 3, k | n} A(n/k) , \\
A(n) &= S_1(n-1) + S_2(n-1) + S_{\geq 3}(n-1).
\end{align*}
These together with appropriate boundary conditions determine all the functions and hence the total number $S(n)$ of singleton cohorts, $S(n) = S_1(n) + S_2(n) + S_{\geq 3}(n)$. Note that the actual recurrences really just involve $S_{\geq 3}$ and $A$ as follows:
\begin{align*}
S_{\geq 3}(n) &= \sum_{k \geq 3, k | n} A(n/k) , \\
A(n) &= \left\{ \begin{array}{ll} S_{\geq 3}(n-1) + S_{\geq 3}(n-2) + S_{\geq 3}((n-3)/2) & \mbox{$n$ odd} \\ S_{\geq 3}(n-1) + S_{\geq 3}(n-2)  & \mbox{$n$ even} \end{array} \right. .
\end{align*}

It might be possible to derive from the above 
some information on the ``average behaviour'' of $S(n)$, 
the number of singleton cohorts of $n$-arch systems. 
But this would likely involve tricky computations with number theoretic arguments, 
that we leave aside for the moment. 

\subsection{Comparing avoidance classes between cohorts}
\label{subsec:comparingGFofCohorts}

One (maybe the most important) purpose of this subsection is to prove that 
the main cohort is the largest in terms of the generating function associated with $\Av(X)$, for $X$ in this cohort. 
This claim is proved as a consequence of more general statements, that allow the comparison of 
such generating functions associated with various cohorts. 

Let us start by introducing some notation. For any cohort $\C$, and any $A$ and $B$ in $\C$, 
we know from Theorem~\ref{th:mainTheorem} that $\Av(A)$ and $\Av(B)$ have the same generating function. 
We may therefore associate this generating function with $\C$ and, when doing so, we denote it $F_{\C}$. 
For two cohorts \C and \D, with generating functions $F_{\C} = \sum c_n t^n$ and $F_{\D} = \sum d_n t^n$, 
we write $\C \leq \D$ when for all $n$, $c_n \leq d_n$. 
We also write $\C < \D$ when $\C \leq \D$ and there exists $n_0$ such that for all $n\geq n_0$ $c_n < d_n$. 
Finally, for any arch system $A$, we denote by $\C_A$ the cohort containing $A$, that is to say the equivalence class of $A$ for $\sim$. 

Variations on the bijective proofs of cases~(\ref{eq:archOverEquivalence}), (\ref{eq:substituteAtoms}) 
and the specialisation of case~(\ref{eq:binaryRotation}) of Theorem~\ref{th:mainTheorem} 
allow us to provide some recursive rules for the comparison of cohorts $\C_A$. 

\begin{proposition}
\label{pr:archOverComparison}
For any arch systems $A$ and $B$, 
if $\C_A \leq \C_B$ then $\C_{\archOverSubscript{A}} \leq \C_{\archOverSubscript{B}}$,
and 
if $\C_A < \C_B$ then $\C_{\archOverSubscript{A}} < \C_{\archOverSubscript{B}}$.

\end{proposition}

\begin{proof}
To prove that $\C_{\archOverSubscript{A}} \leq \C_{\archOverSubscript{B}}$ (resp. $\C_{\archOverSubscript{A}} < \C_{\archOverSubscript{B}}$) 
we should compare (term by term) the enumeration sequences of $\Av(\archOver{A})$ and $\Av(\archOver{B})$, proving that the latter is weakly (resp. eventually strictly) larger. 
To do that, it is enough to give a size-preserving injection (resp. size-preserving injection which fails to be surjective in any size from some $n_0$) 
from  $\Av(\archOver{A})$ to $\Av(\archOver{B})$ given one from $\Av(A)$ to $\Av(B)$.
This follows immediately from the same arguments used in the proof of case (\ref{eq:archOverEquivalence}) of Theorem \ref{th:mainTheorem}, 
essentially by replacing ``bijection'' wherever it occurs by ``injection'' (resp. ``injection which is not surjective in any size from some $n'_0$'' -- observe that $n_0 = n'_0 +1$).
\end{proof}

%

\begin{proposition}
\label{pr:substitutionComparison}
For any arch system $A$ and any atom $b$, if $\C_A \leq \C_b$ then $\C_{PAQ} \leq \C_{PbQ}$, 
and unless $A = a$ is an atom such that $a \sim b$, $\C_{PAQ} < \C_{PbQ}$. 
Moreover, if $\C_A < \C_b$ then $\C_{PAQ} < \C_{PbQ}$. 
\end{proposition}

\begin{proof}
To prove $\C_{PAQ} \leq \C_{PbQ}$, we describe a size-preserving injection from $\Av(PAQ)$ to $\Av(PbQ)$, based on one from $\Av(A)$ to $\Av(b)$.

With the same decomposition used in the proof of case (\ref{eq:substituteAtoms}) of Theorem \ref{th:mainTheorem}, 
we see that, given an injection from $\Av(A)$ to $\Av(b)$, an injection from $\Av(PAQ)$ to $\Av(PbQ)$ can be constructed. 
This uses the fact that if a concatenation $I_1 I_2 \ldots I_k$ of arch systems avoids $A$, then each arch system $I_i$ must avoid $A$. 

If $\C_A < \C_b$, this injection cannot possibly be a bijection (except for the first few sizes $n \leq$ some $n_0$). 
Indeed, it is easy to construct elements of any size $n + |P| + |Q|$ of $\Av(PbQ)$ that do not lie in its image 
from elements of $\Av(b)$ of size $n$ that do not lie in the image of the original injection. 
In fact, for this injection to be a bijection, we need two conditions. 
The first one is that a concatenation of arch systems should avoid $A$ if and only if each arch system in this sequence avoids $A$: this happens exactly when $A$ is an atom. 
The second condition is that the injection from $\Av(A)$ to $\Av(b)$ needs to be a bijection, i.e.~that $A \sim b$. 
\end{proof} 

Propositions~\ref{pr:archOverComparison} and \ref{pr:substitutionComparison} 
are enough to prove that the main cohorts $\bigclass_n = \C_{N_n}$ are the largest 
in the sense that their generating functions $F_{\bigclass_n}$ eventually dominate the generating functions of any other cohort of arch systems of size $n$.
Recall that $N_n$ is the arch system consisting of $n$ nested arches. 

\begin{proposition}
\label{pr:easiestToAvoid}
For every arch system $A$ of size $n$, either $A$ is in the cohort of $N_n$ or $\C_{A} < \C_{N_n}$. 
\end{proposition}

\begin{proof}
The proof is by induction. 
The base case ($n=1$) is clear. 
So assume that $n \geq 2$ and that the statement holds for all $n' < n$. 
Consider an arch system $A$ of size $n$. Either $A = \archOver{X}$ or $A=Xa$ where $a$ is an atom and $X$ a non empty arch system. 

In the first case, by induction we know that exactly one of the following holds:
\begin{itemize}
 \item $X$ is in the cohort of $N_{n-1}$; and then $A$ is in the cohort of $N_n$ by rule (\ref{eq:archOverEquivalence}).
 \item $\C_{X} < \C_{N_{n-1}}$; but then Proposition~\ref{pr:archOverComparison} ensures that $\C_{A} < \C_{N_n}$.
\end{itemize}

In the second case, denoting the size of $X$ by $j$, we know that either $X$ is in the cohort of $N_j$ or $\C_{X} < \C_{N_j}$. 

Assume first that $X \sim N_j$.  If $X$ is an atom, then $Xa \sim N_j a$ by rule (\ref{eq:substituteAtoms}). 
Now either $a \sim N_{n-j}$, in which case $N_j a \sim N_j N_{n-j} \sim N_n$ so that $A = Xa$ is in the cohort of $N_n$; 
or $\C_{a} < \C_{N_{n-j}}$, and Proposition~\ref{pr:substitutionComparison} ensures that 
$\C_A = \C_{Xa} < \C_{XN_{n-j}} \leq \C_{N_j N_{n-j}}$ (using Proposition~\ref{pr:substitutionComparison} again, since $\C_{X} \leq \C_{N_j}$ by induction). 
We conclude using $\C_{N_j N_{n-j}} = \C_{N_n}$. 

If $X$ is not an atom, we deduce from $X \sim N_j$ that $\C_{X} \leq \C_{N_j}$ 
and Proposition~\ref{pr:substitutionComparison} (applied twice) and induction ensure that $\C_{Xa} < \C_{N_j a} \leq \C_{N_j N_{n-j}} = \C_{N_n}$. 

The last case is $\C_{X} < \C_{N_j}$, in which case Proposition~\ref{pr:substitutionComparison} gives $\C_{Xa} < \C_{N_ja} \leq \C_{N_n}$ (as before). 
\end{proof}

Finally, the bijective proof of the specialisation of case~(\ref{eq:binaryRotation}) of Theorem~\ref{th:mainTheorem} 
can also be adapted to the comparison of cohorts. 

\begin{proposition}
\label{pr:comparisonArchSwitching}
For any arch system $A$, and any arch system $b$ which is an atom or empty, $\C_{A\archOverSubscript{b}} \leq \C_{\archOverSubscript{bA}}$. 
Moreover, unless $A$ is an atom, $\C_{A\archOverSubscript{b}} < \C_{\archOverSubscript{bA}}$.
\end{proposition}

\begin{proof}
Let us assume that $A$ is not empty, otherwise the statement is clear. 
Again, we use the same decomposition as in the proof of the specialisation of case (\ref{eq:binaryRotation}) of Theorem \ref{th:mainTheorem} 
to see that an injection from $\Av(A\archOver{b})$ to $\Av(\archOver{bA})$ can be constructed. 

More precisely, the arch systems of $\Av(A\archOver{b})$ either avoid $b$ or are of the form 
\[
A_p \archOver{A_{p-1} \archOver{\ldots \archOver{A_1 \archOver{M \archOver{C} B_0} B_1} \ldots} B_{p-1}} B_p \quad \text{(resp. } A_1 \archOver{M} \text{ if } b \text{ is empty)}
\]
where $C$ contains the contents of $b$ but avoids $b$, 
the concatenation of arch systems $A_p \ldots A_1$ avoids $A$, 
every $B_i$ avoids $b$, 
and the concatenation of arch systems $A_p \ldots A_1 M$ avoids $A\archOver{b}$. 
This last condition implies that $M$ avoids $A\archOver{b}$, but is more restrictive in general. 
It is equivalent exactly when $A$ is an atom (given that $A_p \ldots A_1$ avoids $A$).

On the other hand, the arch systems of $\Av(\archOver{bA})$ either avoid $b$ or are of the form
\[
B_0 \archOver{B_p \archOver{B_{p-1} \archOver{\ldots \archOver{B_{1} \archOver{C} A_1} \ldots} A_{p-1}} A_p} M' \quad \text{(resp. } \archOver{A_1} M' \text{ if } b \text{ is empty)}
\]
where $C$ contains the contents of $b$ but avoids $b$, 
the concatenation of arch systems $A_p \ldots A_1$ avoids $A$, 
every $B_i$ avoids $b$, and $M'$ avoids $\archOver{bA}$ (without further restriction on $M'$). 

So ``mapping the blocks'' recursively as in the proof of the specialisation of case (\ref{eq:binaryRotation}) of Theorem~\ref{th:mainTheorem} 
we get a size-preserving injection from $\Av(A\archOver{b})$ to $\Av(\archOver{bA})$. 
If $A$ is not an atom, we claim that starting at some size $n_0$, this injection is not surjective. 
Indeed, there exist arch systems $M$ of all sufficiently large sizes 
such that $M$ avoids $A\archOver{b}$ but $A_p \ldots A_1 M$ contains $A\archOver{b}$ for some $A_i$ such that $A_p \ldots A_1$ avoids $A$. 
\end{proof}

\section{Conclusions and open problems}

Several questions are left open in this work. 
An important one is certainly to provide a completely bijective proof of our main result (Theorem~\ref{th:mainTheorem}), 
that is: proving case~\ref{eq:binaryRotation} of this theorem bijectively. Even a sensible combinatorial explanation of the rather tidy expression  for $F_{a \archOverSubscript{bc}}$ in terms of $F_a$, $F_b$ and $F_c$ would represent progress in this direction.
Another problem is to prove that the main cohort is the largest also in terms of number of elements it contains. 

But the most intriguing problem is certainly to prove a converse statement to our main theorem: 
that not only does $\sim$ refine Wilf-equivalence but also coincides with it. 
This is stated as Conjecture~\ref{co:cohortDeterminesGF} at the beginning of our paper, 
and we offer a stronger version of this conjecture, by way of conclusion. 

\begin{conjecture}
For any two arch systems $A$ and $B$, both with $n$ arches, 
either $A$ and $B$ are in the same cohort (i.e. $A \sim B$), or the enumeration sequences of $\Av(A)$ and $\Av(B)$ differ at the latest at size $2n-2$.
\end{conjecture}

We have been able to check that this stronger conjecture holds up to arch systems $A$ and $B$ of size 15. 
We further know that the size $2n-2$ is the smallest one for which such a conjecture could be true. 
Indeed, we have identified families of arch systems $A_n$ and $B_n$ of any size $n \geq 4$ 
such that the enumeration sequences of $\Av(A_n)$ and $\Av(B_n)$ coincide up to size $2n-3$ but differ at $2n-2$. 
These are described below. 

Let $\emptyArch^k$ denote the concatenation of $k$ empty arches. Now, for any $n \geq 4$, set $C_n = \archOver{\emptyArch^{n-4}}$, $A_n = \archOver{C_n \emptyArch \emptyArch}$, and $B_n = \emptyArch \emptyArch  \archOver{C_n}$. 
We claim that there is a size preserving bijection between $\Av(A_n)$ and $\Av(B_n)$ restricted to arch systems with at most $2n-3$ arches, 
but that there are more arch systems of size $2n-2$ avoiding $A_n$ than $B_n$. 

Observe that $A_n = \archOver{bA}$ and $B_n = A \archOver{b}$ for $b = C_n$ and $A = \emptyArch \emptyArch$. 
So the proof of Proposition~\ref{pr:comparisonArchSwitching} provides an injection $\varphi$ from $\Av(B_n)$ to $\Av(A_n)$. 
It is relatively easy to see that $\varphi$ is actually a bijection when restricted to arch systems with at most $2n-3$ arches. 
This essentially amounts to examining where these at most $2n-3$ arches can be in arch systems containing $C_n$ but avoiding $A_n$. 
It is also not hard to see that the arch system $\archOver{C_n \emptyArch} \archOver{C_n \emptyArch}$ of size $2n-2$ avoids $A_n$ but is not in the image of $\varphi$. 

To the best of our knowledge, this work is the first global approach to the study of Wilf-equivalences, 
a popular topic of research in the field of permutation patterns from its early days until now -- and arguably so in the wider context of hereditary classes of combinatorial structures. 
It is performed in the context of Catalan structures, or equivalently permutations avoiding 231 and another pattern $\pi$ -- 
which we could call \emph{principal subclasses} of $\Av(231)$. 
We believe that similar investigations, aiming at classifying all Wilf-equivalences between principal subclasses of (well-behaved) permutation classes
should be carried out. One promising example being considered by the first author, Cheyne Homberger and Jay Pantone is the class of separable permutations, $\Av(2413, 3142)$. This comment is motivated in part by the results of \cite{AAVSeparable} which provide a partial parallel of Proposition \ref{pr:classEnumeration} but more generally because the separable permutations permit several other ``well-structured'' representations.

We can even hope to extend our ideas further, to a partial classification of Wilf-equivalences between principal permutation classes, 
i.e.~classes of permutations defined by the avoidance of a single pattern. 
The framework of matchings with excluded sub-matchings, as defined in~\cite{Jelinek:Dyck}, could provide a good tool for that. 
Matchings are similar to arch systems, but were arches are allowed to cross. 
Namely, a matching of size $n$ is a set of $n$ arches connecting $2n$ points arranged along a baseline, 
with all arches above the baseline. 
Obviously, our families $\Av(A)$ of arch systems avoiding a given arch system $A$ can be seen as matchings with excluded sub-matchings: 
namely, those avoiding $\emptyArch \hspace{-8pt} \emptyArch$ and $A$. 
But (principal) permutation classes $\Av(\pi)$ can also be represented as matchings with excluded sub-matchings. 
Indeed, permutations are in immediate correspondence with matchings having all their arches opened before any arch is closed, 
or equivalently with matchings avoiding $\archOver{\emptyArch}$. 
Under this correspondence, a permutation class $\Av(\pi)$ is simply the class of matchings avoiding $\archOver{\emptyArch}$ and the matching encoding $\pi$. 
If it were possible to adapt our work to such cases, 
and in particular to provide an upper bound on the asymptotic number of Wilf-equivalence classes of principal permutation classes, 
this would be a major achievement in the field. 

\subsection*{Acknowledgements} 

Much of the work in this paper was supported by the software suite \emph{PermLab} \cite{PermLab1.0}, and extensions of it. 
Generating function computations were carried out in \emph{Mathematica} \cite{Mathematica}. 

We are grateful to Cheyne Homberger for pointing out references~\cite{Rudolph:Pattern} and~\cite{Chua:Equipopularity} to our attention. 

Mathilde Bouvel would like to thank the Department of Computer Science of the University of Otago for their hospitality and support 
in January and February of 2014 when most of this work was carried out.

\end{document}